\documentclass[10pt,conference]{ieeeconf}
\IEEEoverridecommandlockouts                              
\usepackage{url}

\usepackage{epsfig,amssymb,latexsym,color,amsmath,pifont,colordvi,multicol,verbatim}
\usepackage{enumerate}
\usepackage{multirow}
\usepackage{amsmath, amssymb, bbm, xspace}

\newcommand{\Real}{\ensuremath{\mathbb{R}}}

\DeclareMathOperator*{\st}{subject\;to}

\def\spose#1{\hbox to 0pt{#1\hss}}

\def\text #1{\hbox{\quad#1\quad}}

\def\nthinsp{\mskip -2   mu}

\def\superstar{^{\raise 0.5pt\hbox{$\nthinsp *$}}}
\def\SUPERSTAR{^{\raise 0.5pt\hbox{$*$}}}

\def\lamstarT {\lambda^{\raise 0.5pt\hbox{$\nthinsp *$}T}}

\def\Fscr{{\mathcal F}}

\def\Nscr{{\cal N}}

\def\Nscr{{\cal N}}

\def\hbar{\skew{4.2}\bar h}

		\def\bkE{\mathbb{E}}
		\def\bk1{{\rm 1\kern-.17em l}}
		\def\bkD{{\rm I\kern-.17em D}}
		\def\bkR{{\rm I\kern-.17em R}}
		\def\bkP{{\rm I\kern-.17em P}}
		\def\bkY{{\bf \kern-.17em Y}}
		\def\bkZ{{\bf \kern-.17em Z}}

		\def\beq{\begin{eqnarray}}
		\def\bc{\begin{center}}
		\def\be{\begin{enumerate}}
		\def\bi{\begin{itemize}}
		\def\bs{\begin{small}}
		\def\bS{\begin{slide}}
		\def\ec{\end{center}}
		\def\ee{\end{enumerate}}
		\def\ei{\end{itemize}}
		\def\es{\end{small}}
		\def\eS{\end{slide}}
		\def\eeq{\end{eqnarray}}

	\def\cp2problem#1#2#3#4{\fbox
		 {\begin{tabular*}{0.9\textwidth}
			{@{}l@{\extracolsep{\fill}}l@{\extracolsep{6pt}}l@{\extracolsep{\fill}}c@{}}
				#1 & & $#4 $ 
			\end{tabular*}}}

		\renewcommand{\emph}[1]{\textbf{#1}}

		\def\bk1{{\rm 1\kern-.17em l}}
		\def\bkD{{\rm I\kern-.17em D}}
		\def\bkR{{\rm I\kern-.17em R}}
		\def\bkP{{\rm I\kern-.17em P}}
		
		\def\bkZ{{\bf{Z}}}

\newcommand {\beeq}[1]{\begin{equation}\label{#1}}
\newcommand {\eeeq}{\end{equation}}
\newcommand {\bea}{\begin{eqnarray}}
\newcommand {\eea}{\end{eqnarray}}

\def\texitem#1{\par\smallskip\noindent\hangindent 25pt
               \hbox to 25pt {\hss #1 ~}\ignorespaces}

\def\st{\mbox{subject to}}

\newtheorem{assumption}{Assumption}
\newtheorem{lemma}{Lemma}

\newtheorem{proposition}{Proposition}

\def\argmin{\mathop {\rm argmin}}
\def\a{{\alpha}}
\def\g{{\gamma}}
\def\l{{\lambda}}

\newcommand{\captionfonts}{\small}
\makeatletter  
\long\def\@makecaption#1#2{%
  \vskip\abovecaptionskip
  \sbox\@tempboxa{{\captionfonts #1: #2}}%
  \ifdim \wd\@tempboxa >\hsize
    {\captionfonts #1: #2\par}
  \else
    \hbox to\hsize{\hfil\box\@tempboxa\hfil}%
  \fi
  \vskip\belowcaptionskip}
\makeatother   

\newcommand{\us}[1]{\textcolor{black}{#1}}

\begin{document}

\title{\Large \bf Distributed Stochastic Optimization under Imperfect Information}
\author{Aswin Kannan \ \and \  Angelia Nedi\'{c} \ \and \ Uday V.~Shanbhag\thanks{Kannan and Shanbhag are at 
the Dept. of Indust.\ and Manuf.\ Engg.\ at the Penn.\ State Univ., University Park, PA, 16803.
	Nedi\'{c} is at the Indust.\ and Enterprise Sys.\ Engg.\ Dept.,
	Univ.\ of Illinois at Urbana-Champaign,
Urbana, IL, 61801. They are reachable at
{\{axk1014,udaybag\}}@psu.edu and angelia@illinois.edu. Kannan and
Shanbhag's research
has been partially funded by
NSF CAREER award CMMI-1246887 and CMMI-1400217 while
Nedi\'c's research is funded by NSF awards CCF 11-11342 and DMS 13-12907, and
ONR grant No.\ 00014-12-1-0998.}}

\maketitle
\begin{abstract}
We consider a stochastic convex optimization problem that requires minimizing a sum
of misspecified agent-specific expectation-valued convex functions over
the intersection of a collection of agent-specific convex sets. 
This misspecification is manifested in a  parametric sense and  may
be resolved through solving a distinct stochastic convex learning
problem. Our interest lies in the development of distributed algorithms
in which every agent makes decisions based on the knowledge of its objective and 
feasibility set while learning the decisions of other agents by
communicating with its local neighbors over a time-varying
connectivity graph. While a significant body of research currently
exists in the context of such problems, we believe that the misspecified
generalization of this problem is both important and has seen little
study, if at all. Accordingly, our focus lies on the simultaneous resolution of both problems
through a joint set of schemes that combine three distinct steps: (i) An
alignment step in which every  agent updates its current  belief by
averaging over the beliefs of its neighbors; (ii) A projected (stochastic) gradient step in which
every agent further updates this averaged estimate; and (iii) A learning
step in which agents update their belief of the misspecified parameter
by utilizing a stochastic gradient step. Under an assumption of mere
convexity on agent objectives and strong convexity of the learning
problems,  we show that the sequences generated by this collection of
update rules converge almost surely to the solution of the correctly
specified stochastic convex optimization problem and the stochastic
learning problem, respectively. 
\end{abstract}

\section{Introduction}
Distributed algorithms have grown enormously in relevance for addressing
a broad class of problems in arising in network system applications in
control and optimization, signal processing, communication networks,
power systems, amongst others
(c.f.~\cite{metric12,garcia12cdc,jianshu13}). A crucial assumption in any
 such framework is the need for precise specification of the objective
 function. In practice however, in many engineered and economic systems,
 agent-specific functions may be misspecified from a parametric
 standpoint but may have access to observations that can aid in
 resolving this misspecification. Yet almost all of the efforts in
 distributed algorithms obviate the question of misspecification in the
 agent-specific problems, motivating the present work.

In seminal work by Tsitsiklis~\cite{tsit84thesis}, 
decentralized and distributed approaches to decision-making and optimization
were investigated in settings complicated by
partial coordination, delayed communication, and the presence of noise. 
In subsequent work~\cite{tsit86}, the behavior
of general distributed gradient-based algorithms was examined. In related work on parallel
computing~\cite{Cybenko89}, iterative approaches and their convergence rate
estimates were studied for distributing computational load amongst multiple processors. 

Consensus-based extensions to optimization with linear constraints
were considered in~\cite{distributeddual87}, while convergent algorithms for
problems under settings of general agent specific convex constraints
were first proposed in~\cite{nedich10consensus}, as an extension of distributed multi-agent model 
proposed in~\cite{NO2009}, and further developed in~\cite{ramdisdata12}.  
In~\cite{deming11}
and~\cite{zhu10}, a problem with common (global)
inequality and equality constraints is considered 
and distributed primal-dual projection method is proposed. A more general case
with agents having only partial information with respect to shared and
nonsmooth constraints is studied in~\cite{tsung14}. Recent
work~\cite{zanella12} compares and obtains rate estimates for Newton and
gradient based schemes to solve distributed quadratic minimization, a
form of weighted least squares problem for networks with time varying
topology.  In~\cite{tsianos13}, a distributed dual-averaging algorithm
is proposed combining push-sum consensus and gradient
steps for constrained optimization over a static graph, 
while in~\cite{NOl2015}, a subgradient method is developed using push-sum
algorithm on time-varying graphs. Distributed algorithms that combine consensus
and gradient steps have been recently developed~\cite{subgram10,matei12}
for stochastic optimization problems.  In recent
work~\cite{leenedich13}, the authors consider a setting of asynchronous
gossip-protocol, while stochastic extensions to asynchronous
optimization were considered in~\cite{ramgossip12,nedic2011,touri10,ram10step},
convergent distributed schemes were proposed, and error bounds for
finite termination were obtained.  All aforementioned work assumes that the functions 
are either known exactly or their noisy gradients are available.

While misspecification poses a significant challenge in the resolution of
optimization problems, general purpose techniques for the resolution of misspecified
optimization problems through the joint solution of the misspecified problem
and  a suitably defined learning problem have been less studied. Our framework
extends prior work on 
deterministic~\cite{ahmadi14} and
stochastic~\cite{jiang13,jiang15solution} gradient
schemes. 
Here, we consider a networked regime in which agents are
characterized by misspecified expectation-valued  convex objectives and
convex feasibility sets. The overall goal lies in minimizing the sum of
the agent-specific objectives over the intersection of
the agent-specific constraint sets. In contrast with traditional models,
agents have access to the stochastic convex learning metric that allows for
resolving the prescribed misspecification. Furthermore, agents only have
access to their objectives and their feasibility sets and may observe the
decisions of their local neighbors as defined through a general time-varying
graph. In such a setting, we considered distributed protocols that
combine three distinct steps: (i) An {\em alignment} step in which every
agent updates its current belief by averaging over the beliefs of its
neighbors based on a set of possibly varying weights; (ii) A  projected
(stochastic) {\em gradient} step in which every agent further updates this
averaged estimate; and (iii) A {\em learning} step where agents update
their belief of the misspecified parameter by utilizing a stochastic
gradient step. We show that the produced sequences of
agent-specific decisions and agent-specific beliefs regarding the misspecified parameter converge  in an
almost sure sense to the optimal set of solutions and the optimal
parameter, respectively under the assumption of general time-varying
graphs and note that this extends the results
in~\cite{nedich10consensus}.

The  paper is organized as follows.  In Section~\ref{sec:problem}, \us{we define the
problem of interest and provide a motivation for its study}.  In
Section~\ref{sec:assump}, we outline our algorithm and the relevant assumptions. Basic properties of the
algorithm are investigated in Section~\ref{sec:algo} and the almost sure
convergence of the produced sequences is established in Section~\ref{sec:global}. We conclude the
paper with some brief remarks in Section~\ref{sec:concl}.

\section{Problem Formulation and Motivation}\label{sec:problem}
We consider a networked multi-agent setting with time-varying undirected connectivity graphs,
where the graph at time $t$ is denoted by $\mathcal{G}^t = \left\{
\mathcal{N},\mathcal{E}^t\right\}$,  $\mathcal{N} \triangleq \left\{1,\ldots,m
\right\}$ denotes the set of nodes and $\mathcal{E}^t$  is the set of
edges at time $t$. Each node represents a single agent and the problem of
interest is
\begin{align}
\label{eq:initprob}
\begin{aligned}
\hbox{minimize}   & \quad \sum_{i=1}^m
\mathbb{E}[\varphi_i(x,\theta^*,{\xi})] \\
\st 	 & \quad x \in \bigcap_{i=1}^m X_i,
\end{aligned}
\end{align}
where $\theta^*\in\Real^{p} $ represents the (misspecified) vector of
parameters, $\mathbb{E}[\varphi_i(x,\theta^*,\xi)]$ denotes
the local cost function of agent $i$, the expectation is taken with
respect to a random variable $\xi$, defined as  ${\xi}:\Omega \to
\Real^d$, and $(\Omega, {\mathcal F}, \mathbb{P})$ denotes the
associated probability space.  The function $\varphi_i: \Real^n \times
\Real^p \times \Real^d {\to \Real}$ is assumed to be convex and
continuously differentiable in $x$ for all $\theta\in{\Theta}$ and
all $\xi \in \Omega$. \\
\noindent (i) {\bf Local information}. 
Agent $i$ has access to
its objective function $\mathbb{E}[\varphi_i(x,\theta^*,\xi)] $  
and its set $X_i$ but is unaware of the objectives and
	constraint sets of the other agents. Furthermore, it may
	communicate at time $t$ with its local neighbors, as specified
		by the graph
	$\mathcal{G}^t$; \\
\noindent (ii) {\bf Objective misspecification}. 
The agent objectives are parametrized by a vector
$\theta^*$ unknown to the agents.

We assume that the true parameter $\theta^*$ is a solution to a
	distinct convex problem, accessible to every agent:
\begin{align}
 \min_{\theta \in \Theta} \, \mathbb{E}[g(\theta,\chi)],
\label{eq:thetaprob}
\end{align}
where $\Theta\subseteq\Real^p$ is a closed and convex set, 
$\chi: \Omega_{\theta}
\to \Real^r$ is a random variable with the associated probability space given by
$(\Omega_{\theta}, \Fscr_{\theta},
		\mathbb{P}_{\theta})$, while
$g: \Real^p \times \Real^r \to \Real$ is a strongly convex and
continuously differentiable function in $\theta$ for every $\chi$.
Our interest lies in the joint solution of~\eqref{eq:initprob}--\eqref{eq:thetaprob}:
\begin{align}\label{eq:prob0}
\begin{aligned}
x^* \in \argmin_x  & \left\{\sum_{i=1}^m
	\mathbb{E} [\varphi_i(x,\theta^*,\xi)] \mid x \in \bigcap_{i=1}^m
		X_i\right\},\cr
\theta^* \in \argmin_\theta & \left\{ \mathbb{E}[g(\theta,\chi)] \mid \theta
\in \Theta\right\}.
\end{aligned}
\end{align}

\noindent {\bf A sequential approach:} Traditionally, such problems are approached sequentially: (1) an
accurate approximation of $\theta^*$ is first obtained; and 
(2)~given $\theta^*$,
standard computational schemes are then applied. However, this
avenue is inadvisable when the learning problems are stochastic and
accurate solutions are available via simulation schemes, requiring significant effort. In fact, if the
learning process is terminated prematurely, the resulting solution
may differ significantly from $\theta^*$ and this error can only be
captured in an expected-value sense. Thus,
such approaches can only provide approximate solutions and, 
consequently, {\em cannot} generally provide asymptotically
exact solutions.  Inspired by recent work on learning and optimization in a centralized
regime~\cite{jiang13}, we consider the development of schemes
for distributed stochastic optimization. We build on \us{the} distributed
projection-based algorithm~\cite{nedich10consensus}, which combines
local averaging with a projected gradient step for agent-based
constraints. In particular, we introduce an additional layer of a
learning step to aid in resolving the misspecification. 

\noindent {\bf Motivating applications:} Consensus-based optimization
problems arise in a range of settings including the dispatch of
distributed energy resources (DERs)~\cite{garcia12cdc}, signal
processing~\cite{}, amongst others. Such settings are often complicated
by misspecification; for instance, there are a host of charging,
   discharging and efficiency parameters associated with storage
   resources that often require estimation.

\section{Assumptions and Algorithm}\label{sec:assump}
We begin by presenting a distributed  framework for solving
the problem in~\eqref{eq:prob0}.
To set this up more concretely, for all $i\in{\cal N}$, we let
$f_i(x,\theta)
	 \triangleq\mathbb{E}[\varphi_i(x,\theta,{\xi})]$
for all $x$ and $\theta\in\Theta$ and  $h(\theta)  \triangleq
\mathbb{E}[g(\theta,\chi)]$ {for all $\theta\in\Theta$}.
Then, problem~\eqref{eq:prob0} assumes the following form:
\begin{align}
x^* & \in \argmin_{x \in \cap_{i=1}^m X_i} \  f(x,{\theta^*}), \mbox{ where }
f(x,\theta) \triangleq \sum_{i=1}^m f_i(x,\theta^*), \notag \\
\theta^* & \in  \argmin_{\theta\in\Theta} \  h(\theta).
\label{eq:prob}
\end{align}
We consider a distributed algorithm where agent $i$ knows $f_i$ and the set $X_i$, while all agents have access to $h$. We further assume
	that $i$th agent has access to oracles that produce  random
		samples 
		$\nabla_{x} \varphi_i(x,\theta,\xi)$ and $\nabla_{\theta} g(\theta,\chi)$.  
The information {needed by agents} to solve the optimization problem
is acquired through local sharing of the estimates over a time-varying communication network.
Specifically, 
at iteration $k$, {the $i$th agent has estimates $x_i^k\in X_i$ and
	$\theta_i^k\in\Theta$ and at the next iteration,} 
 constructs a vector $v_i^k$, as an average of the vectors $x_j^k$
 obtained from its {local} neighbors, given by:
\begin{align}\label{aver}
v_{i}^{k} := \sum_{j=1}^{m}a_{i}^{j,k}x_{j}^{k}\quad
\hbox{for all $i = 1, \hdots, m$ and $k\ge0$},\end{align}
where {the weights} $a_{i}^{j,k}$ are nonnegative scalars
{satisfying}  $\sum_{j = 1}^{m}
a_{i}^{j,k} = 1$ and are related to the underlying
connectivity graph $\mathcal{G}^k$ over which the agents communicate at time $k$.
Then, {for $i = 1, \hdots, N$, the $i$th} agent updates its $x$- and $\theta$-variable as follows:
\begin{align}
x_{i}^{k+1} & := \Pi_{{X_i}} \left( v_{i}^{k} - {\alpha_k} \left( \nabla_x
			f_{i} (v_{i}^{k},\theta_i^{k}) + w_i^{k} \right)\right),
	\label{proj-x}  \\
\theta_i^{k+1} & := \Pi_{\Theta} \left( \theta_{i}^{k} -
		{\gamma_{k}} \left( \nabla h (\theta_{i}^{k}) + \beta_i^{k} \right)\right),
\label{eq:projec}
\end{align}
where
$w_i^k \triangleq  \nabla_x \varphi_i(v_i^k,\theta_i^k,\xi_i^k)-\nabla_x
f_i(v_i^k,\theta_i^k)$ with $\nabla_x f_i(x,\theta) = \mathbb{E}[\nabla_x
\varphi_i(x,\theta,\xi)]$, and 
$\beta_i^k \triangleq \nabla_\theta g
(\theta_{i}^{k},\chi_i^k) - \nabla h(\theta_{i}^{k})$ with
$\nabla h(\theta_i^k) = \mathbb{E}[\nabla_\theta g(\theta_i^k,\chi)]$ for all $i\in{\cal N}$ and all $k\ge0$.
The parameters $\a_k>0$ and $\g_k>0$ {represent} stepsizes at epoch $k$, while the initial points 
$x_i^0\in X_i$ and $\theta_i^0\in\Theta$ are randomly selected for each agent $i$.
The $i$th agent has access  to $\left( \nabla_x
			f_{i} (v_{i}^{k},\theta_i^{k}) + w_i^{k} \right)$
and not $\left( \nabla_x f_{i} (v_{i}^{k},\theta_i^{k}) \right)$. The same is the case with the learning 
function.
At time epoch $k$, agent $i$ proceeds to average
over its neighbors' decisions by using the weights in \eqref{aver} and
employs this average to update its decision in
\eqref{proj-x}.  Furthermore, agent $i$ makes a subsequent update in its
belief regarding $\theta^*$, by taking a similar (stochastic) gradient
update, given by \eqref{eq:projec}.  

The weight $a_{i}^{j,k}$ used by agent $i$ for the iterate of agent $j$ at time $k$
is based on the connectivity graph $\mathcal{G}^k$.
Specifically, letting
$\mathcal{E}^k_{i}$ be the set of neighbors of agent $i$:
\[\mathcal{E}^k_{i}=\{j\in\mathcal{N}\mid \{j,i\}\in \mathcal{E}^k\}\cup\{i\},\]
the weights $a_{i}^{j,k}$ are compliant with the neighbor structure:
\[ a_{i}^{j,k} > 0 \hbox{ if  $j \in \mathcal{E}^k_{i}$}\quad\hbox{and}\quad
a_{i}^{j,k}= 0 \hbox{ if $j \not \in \mathcal{E}^k_{i}$}. \]
We assume that each graph $\mathcal{G}^k$ is connected and
that matrices are doubly stochastic, as given in the following assumption.
\begin{assumption}[Graph and weight matrices]\label{assump:doubstoch}
\noindent  \\
\noindent (a) The matrix $A(k)$ (whose $(ij)$th entry is denoted by $a_{i}^{j,k}$)  is doubly stochastic for every $k$, i.e.,
$\sum_{i=1}^{m}a_{i}^{j,k} = 1$ for every $j$ and
$\sum_{j=1}^{m}a_{i}^{j,k} = 1$ for every $i$.\\
\noindent 
(b) The matrices $A(k)$ have positive diagonal entries, and all positive entries in every $A(k)$ are uniformly bounded away from zero, i.e., there exists $\eta>0$ such that, for all $i,j,$ and $k$, we have
$a_{i}^{j,k}\ge\eta$ whenever $a_{i}^{j,k}>0$.\\
\noindent 
(c) 
The graph $\mathcal{G}^k$ is connected for every $k \geq 0$. 
\end{assumption}
The instantaneous connectivity assumption on the graphs $\mathcal{G}^k$ can be relaxed by requiring that
the union of these graphs is connected every $T$ units of time, for
instance. The analysis of this case  
is similar to that  given in this paper. We choose to work with connected graphs in order
to keep the analysis somewhat simpler and to provide a sharper focus on the learning aspect of the problem.

Next, we define ${\mathcal F}_0 \triangleq (x_i^0,\theta_i^0),\ i\in\mathcal{N}\}$ and
${\mathcal F}_k =\{(\xi_i^t,
   \chi_i^t), \ i\in\mathcal{N},\ t=0,1,\ldots,k-1\}$ for all $k\ge 1$
and make the following
   assumptions on the {conditional} first and second moments of
   the stochastic errors $w_i^k$ and $\beta_i^k$. These assumptions
	   are relatively standard in the development of stochastic gradient
		   schemes. 
\begin{assumption}
[Conditional first and second moments\\]
\label{assump:errormoment}
(a) $\bkE[w_{i}^{k}\mid \mathcal{F}_{k}]=0$ and
$\bkE[\beta_{i}^{k}\mid \mathcal{F}_{k}]=0$  for all $k$ and $i \in {\cal N}$.\\
(b)  $\bkE[\| w_i^{k} \|^2\mid \mathcal{F}_{k}]\le \nu^2$ and  
$\bkE[\| \beta_i^{k} \|^2 \mid \mathcal{F}_{k}]\le \nu_{\theta}^2$ for all $k$ and $i \in {\cal N}$.\\
$($c) $\mathbb{E}[\|\theta_i^0\|^2]$ is finite for all $i\in {\cal N}$.
\end{assumption}
We now discuss the assumptions {on agent objectives}, 
the learning metric and the underlying set constraints.

\begin{assumption}[{Feasibility sets}\\]\label{assump:bnd}
(a) For every $i \in \Nscr$, the set $X_i\subset\Real^n$ is convex and compact.\\
(b) The intersection set $\cap_{i=1}^m X_i$ is nonempty.\\ 
(c) The set $\Theta\subseteq\Real^p$ is convex and closed.
\end{assumption}
Note that under the compactness assumption on the sets $X_i$, we have  
$\mathbb{E}[\|x_i^0\|^2]<\infty$ for all $i\in {\cal N}$. Furthermore we have
\begin{align}
\label{assump-bnd}
\max_{x_i,y_i \in X_i} \|x_i -y_i\| \leq D
\end{align}
for some scalar $D>0$ and for all $i$.
Next, we consider the conditions for the agent objective functions.

\begin{assumption}
[Agent objectives\\]
\label{assump:functionsf}
 For every $i \in \mathcal{N}$, the function $f_i(x,\theta)$ is convex in $x$ for every
	$\theta\in\Theta$. Furthermore, for every $i \in \mathcal{N}$, the gradients $\nabla_{x}
	f_i(x,\theta)$ are uniformly Lipschitz continuous functions in $\theta$ for all
		$x\in X_i$: $\|\nabla_{x}f_i(x,\theta^a)-\nabla_{x} f_i(x,\theta^b)\|
		\leq L_{\theta} \|\theta^a-\theta^b \|$ for all $\theta^a,\theta^b\in\Theta$, all $x\in X_i$, and all $i \in
		\Nscr$.
\end{assumption}

\begin{assumption}
[Learning metric\\]
\label{assump:functionsh}
The function $h$ is strongly convex over $\Theta$ with a constant $\kappa>0$, and 
its gradients are Lipschitz continuous with a constant $R_{\theta}$, i.e.,
$\| \nabla h(\theta^{a})-\nabla h(\theta^{b}) \| \leq R_{\theta} \|\theta^a-\theta^b \|$ 
for all $\theta^a, \theta^b \in \Theta$.
\end{assumption}

By the strong convexity of $h$, 
the problem \eqref{eq:thetaprob} has a unique solution denoted by $\theta^*$.  
From the convexity of the functions $f_i$ in $x$ (over $\Real^n$) for every
	$\theta\in\Theta$, as given in Assumption~\ref{assump:functionsf}, these functions are continuous. Thus, {when $\cap_{i=1}^m X_i$ is nonempty and each $X_i$ is compact (Assumption~\ref{assump:bnd}),
	the problem $\min_{x\in\cap_{i=1}^m X_i} \sum_{i=1}^m
	f_i(x,\theta^*)$ has a solution.}
	
\section{Basic properties of the algorithm}\label{sec:algo}
In this section, we provide some basic relations for the
algorithm~\eqref{aver}--\eqref{eq:projec} that are fundamental to
establishing the almost sure convergence of the sequences produced by
the algorithm. The proofs of all the results can be found in~\cite{KNU2015}.

\subsection{Iterate Relations}
We start with a simple result for weighted averages of a finitely many points.
\begin{lemma}\label{lma:aver}
Let $y_1,\ldots,y_m\in\Real^n$ and $\l_1,\ldots,\l_m\in\Real$, with $\l_i\ge0$ for all $i$ and $\sum_{i=1}^m\l_i=1$.
Then, for any $c\in\Real^n$, we have
\[\left\|\sum_{i=1}^m\l_i y_i -c \right\|^2 
=\sum_{i=1}^m\l_i\|y_i-c\|^2 -\frac{1}{2}\sum_{j=1}^m\sum_{\ell=1}^{m}\l_j\l_\ell\|y_j - y_\ell\|^2.\]
\end{lemma}
\begin{proof}
By using the fact that $\l_i$ are convex weights, we write
\begin{align*}
\left\|\sum_{i=1}^m\l_i y_i -c \right\|^2
&=\left\|\sum_{i=1}^m\l_i (y_i -c) \right\|^2 \\
&=\sum_{i=1}^m \sum_{j=1}^m\l_i\l_j (y_i -c)^T(y_j -c).\end{align*}
Noting that $2a^Tb=\|a\|^2 +\|b\|^2 - \|a-b\|^2$, valid for any $a,b\in\Real^n$, and applying it to each inner product,
we obtain
\begin{align*}
& \quad \left\|\sum_{i=1}^m\l_i y_i -c \right\|^2 \\
& =\frac{1}{2}\sum_{i=1}^m \sum_{j=1}^m\l_i\l_j 
\left( \|y_i -c\|^2 + \|y_j -c\|^2 -\|y_i-y_j\|^2\right)\cr
&=\sum_{i=1}^m \l_i\|y_i -c\|^2 -\frac{1}{2}\sum_{i=1}^m \sum_{j=1}^m\l_i\l_j \|y_i -y_j\|^2,
\end{align*}
where the second equality follows by noting that 
$\frac{1}{2}\sum_{i=1}^m \sum_{j=1}^m\l_i\l_j 
\left( \|y_i -c\|^2 + \|y_j -c\|^2\right)=\sum_{p=1}^m \l_p \|y_p -c\|^2$, which can be seen 
by using $\sum_{j=1}^m\l_j=1$.
\end{proof}
We use the following lemma that provides a bound on the difference between consecutive $x$-iterates of the algorithm and an analogous relation 
for consecutive $\theta$-iterates. 

\begin{lemma}\label{lemm:xdesc}
Let Assumptions~\ref{assump:doubstoch}--\ref{assump:functionsf} hold.
Also, let $X=\cap_{i=1}^m X_i$ and let $h$ be strongly convex over $\Theta$.
Let the iterates $x^k_i$ be generated
according to~\eqref{aver}--\eqref{eq:projec}.  Then, almost surely, we have for all $x\in X$ and $k\ge 0$,
\begin{align*}
& \sum_{i = 1}^{m}  \bkE  \left[ \|x_{i}^{k+1} -x \|^2 \mid \mathcal{F}_k\right]  
\leq \sum_{j = 1}^{m} \|x_{j}^k -x\|^2  \cr
& - \eta^2 \sum_{\{s,\ell\}\in \mathcal{T}^k}  \|x_s^k-x_\ell^k\|^2
+m\alpha_k^2 (2S^2 + \nu^2)
 + m\alpha_k^{2-\tau}L_{\theta}^2 D^2\\
 & + \left(\alpha_k^{\tau} + 2\alpha^2_kL_{\theta}^2\right) 
\sum_{i =1}^{m}\|\theta_i^k-\theta^*\|^2\\
& -2\alpha_k \sum_{i = 1}^{m}\left(f_i(v_i^k,\theta^*)-f_i(x,\theta^*)\right),
\end{align*}
where $\mathcal{T}^k$ is a spanning tree in the graph $\mathcal{G}^k,$
$\tau \in (0,2)$ is an arbitrary but fixed scalar,  
$\theta^*=\argmin_{\theta\in\Theta} h(\theta)$, and 
$S=\max_i\max_{x\in \bar X}\| \nabla_x f_i(x,\theta^*) \|$, 
with $\bar X$ being the convex hull of the union $\cup_{i=1}^m X_i$.
\end{lemma}
\begin{proof}
First, we note that by the strong convexity of $h$, 
the point $\theta^*\in\Theta$ minimizing $h$ over $\Theta$ exists and it is unique.
Next, we use the projection property for a closed convex set $Y$, according to which
we have,
 $ \|\Pi_Y[x] - y\|^2 \leq \|x-y\|^2$ for all $y \in Y$ and all $x$.
 Therefore, for all $i$, and any $x\in X$,  we have
\begin{align}\label{eq:initdesc} 
    & \quad \|  x_{i}^{k+1} -x \|^2 \nonumber \\
     & = \left\|\Pi_{X_i}\left(v_{i}^{k} -
		  \alpha_k \left( \nabla_x f_{i} (v_{i}^{k},\theta_i^k) + w_i^{k}
			  \right)\right) - x  \right\|^2  \nonumber \\
 & \leq \left\|v_{i}^{k} - \alpha_k \left( \nabla_x f_{i} (v_{i}^{k},\theta_i^k) + w_i^{k} \right) - x \right\|^2\cr
 & \leq \|v_{i}^{k} - x \|^2 + T_A^k  + T_B^k, \\ 
 \mbox{ where }  T_A^k  & \triangleq {\alpha_k^2} \|\nabla_x f_i(v_i^k,\theta_i^k)  + w_i^k \|^2,
	\label{defTA} \\
		T_B^k  & \triangleq  -2\alpha_k(v_i^k-x)^T(\nabla_x f_i(v_i^k,\theta_i^k) +
				w_i^k).\label{defTB}
\end{align}
Expanding $T_A^k$, we obtain that
\begin{align*}
 T_A^k  & 
 = \alpha_k^2 \|\nabla_x f_i(v_i^k,\theta_i^k)  + w_i^k \|^2 
 = \alpha_k^2 \|\nabla_x f_i(v_i^k,\theta_i^k)  + w_i^k \|^2 \\
& = \alpha_k^2 \| \nabla_x f_i(v_i^k,\theta_i^k) \|^2 + \alpha_k^2  \|
w_i^k \|^2 \\ & + 2 \alpha_k^2  (w_i^k)^{T} \nabla_x f_i(v_i^k,\theta_i^k ).
\end{align*}
Taking the conditional expectations on both sides of \eqref{exp-Ta} with respect to the past
and using Assumption~\ref{assump:errormoment} on the stochastic gradients, we have that almost surely,
\begin{align}\label{exp-Ta}
 &\quad \bkE \left[ T_A^k  \mid \mathcal{F}_{k} \right] \nonumber \\
&  = \alpha_k^2 \| \nabla_x f_i(v_i^k,\theta^*) +
	 \nabla_x f_i(v_i^k,\theta_i^k) -  \nabla_x f_i(v_i^k,\theta^*) \|^2
 \nonumber \\ 
& + \alpha_k^2 \underbrace{\bkE \left[ \| w_i^k \|^2 \mid
	\mathcal{F}_{k} \right]}_{\ \leq \ \nu^2}
 + 2\alpha_k^2 \underbrace{\bkE \left[ (w_i^k)^{T} \nabla_x
	 f_i(v_i^k,\theta_i^k ) \mid \mathcal{F}_{k} \right]}_{ \ = \ 0} \cr 
	 &\le \alpha_k^2  ( \| \nabla_x f_i(v_i^k,\theta^*) \|
	 +\|\nabla_x f_i(v_i^k,\theta_i^k) -  \nabla_x f_i(v_i^k,\theta^*) \|)^2 \cr
	 & +\alpha_k^2 \nu^2 \nonumber \\
	 & \leq \alpha_k^2 (2S^2 + 2L_{\theta}^2 \|\theta_i^k-\theta^* \|^2 + \nu^2),
\end{align}
where in the last inequality we use $(a+b)^2\le 2a^2 + 2b^2$ valid for all $a,b\in\Real$, and the Lipschitz property 
of $\nabla_x f_i(x,\theta)$ (cf.\ Assumption~\ref{assump:functionsf}). Furthermore, since the sets $X_i$ 
are compact by Assumption~\ref{assump:bnd}, the convex hull $\bar X$ of $\cup_{i=1}^m X_i$ is also compact,
implying by continuity of the gradients that
$\max_i\max_{k\ge0}\| \nabla_x f_i(v_i^k,\theta^*) \| \le \max_i\max_{x\in \bar X}\| \nabla_x f_i(x,\theta^*) \|=S,$
with $S<\infty$.

Next, we consider the term $T_B^k$. By taking the conditional expectation with respect to $\mathcal{F}_{k}$ 
and using  $\bkE\left[(v_i^k-x)^T  w_i^k \mid \mathcal{F}_{k} \right]=0$, we obtain
\begin{align*}
& \bkE\left[ T_B^k \mid \mathcal{F}_{k} \right] \\
& =  -2\alpha_k(v_i^k-x)^T(\nabla_x f_i(v_i^k,\theta_i^k)-\nabla_x
		f_i(v_i^k,\theta^*)) \\
& -  2\alpha_k(v_i^k-x)^T \nabla_x f_i(v_i^k,\theta^*) \cr
& \leq 2\alpha_k \| v_i^k-x\| \| \nabla_x f_i(v_i^k,\theta_i^k)-\nabla_x f_i(v_i^k,\theta^*) \| \\
& -2\alpha_k(v_i^k-x)^T \nabla_x f_i(v_i^k,\theta^*).
\end{align*}
By using the Lipschitz property 
of $\nabla_x f_i(x,\theta)$ (cf.\ Assumption~\ref{assump:functionsf}), 
the relation  $2\alpha ab=2(\sqrt{\alpha^{2-\tau}} a) (\sqrt{\alpha^{\tau}} b)$ valid for any $a,b\in\Real$ and any $\tau>0$,
and the Cauchy-Schwarz inequality, we further obtain
\begin{align}\label{exp-Tb}
& \bkE\left[ T_B^k \mid \mathcal{F}_{k} \right] \nonumber \\
& \leq 2\alpha_k L_{\theta} \| v_i^k-x\| \| \theta_i^k - \theta^* \|
-2\alpha_k(v_i^k-x)^T \nabla_x f_i(v_i^k,\theta^*)  \cr
& \leq \alpha_k^{2-\tau}L_{\theta}^2 \|v_i^k -x\|^2 
+ \alpha_k^\tau
\|\theta_i^k-\theta^*\|^2 \nonumber \\
& -2\alpha_k(v_i^k-x)^T \nabla_x f_i(v_i^k,\theta^*)  \nonumber \\
& \leq \alpha_k^{2-\tau}L_{\theta}^2 D^2  + \alpha_k^\tau \|\theta_i^k-\theta^*\|^2 \nonumber \\
& -2\alpha_k\left(f_i(v_i^k,\theta^*)- f_i(x,\theta^*)\right),
\end{align}
where in the last inequality we also employ the convexity of $f_i$ and boundedness of sets $X_i$, together 
with the fact that $v_i^k,x\in X_i$ for all $i$ (cf.~Assumption~\ref{assump:bnd}).

Now, we take the conditional expectation in relation~\eqref{eq:initdesc} and we substitute estimates~\eqref{exp-Ta} 
and~\eqref{exp-Tb}, which yields almost surely, for all $i$, all $x\in X$ and all $k$, 
\begin{align*}
& \bkE  \left[ \|x_{i}^{k+1} -x \|^2 \mid \mathcal{F}_k\right]  \\
& \leq \|v_{i}^k -x \|^2 + \alpha_k^2 (2S^2 + \nu^2)
 + \alpha_k^{2-\tau}L_{\theta}^2  D^2\cr
& + \left(\alpha_k^{\tau} + 2\alpha^2_kL_{\theta}^2\right) 
\|\theta_i^k-\theta^*\|^2 \\
&-2\alpha_k \left(f_i(v_i^k,\theta^*)-f_i(x,\theta^*)\right).
\end{align*}
Summing the preceding relations
over $i = 1, \hdots, m$, we have the following inequality almost surely, for all $x\in X$ and all $k\ge0$,
\begin{align}\label{eq:mid0}
& \quad \sum_{i = 1}^{m}  \bkE  \left[ \|x_{i}^{k+1} -x \|^2 \mid \mathcal{F}_k\right]  \cr 
& \leq \sum_{i = 1}^{m} \|v_{i}^k -x\|^2 + m\alpha_k^2 (2S^2 + \nu^2)
 + m\alpha_k^{2-\tau}L_{\theta}^2 D^2\cr
& + \left(\alpha_k^{\tau} + 2\alpha^2_kL_{\theta}^2\right) \sum_{i =
		1}^{m}\|\theta_i^k-\theta^*\|^2 \nonumber \\
& -2\alpha_k \sum_{i = 1}^{m}\left(f_i(v_i^k,\theta^*)-f_i(x,\theta^*)\right).
\end{align}

We now focus on the term $\|v_{i}^k -x\|^2$. Noting that ${\sum_{j=1}^m a_i^{j,k} = 1}$, 
by Lemma~\ref{lma:aver}, it follows that for all $x\in X$ and all $k\ge0$,
\begin{align*}
& \quad \|v_{i}^k -x\|^2
=\left\|\sum_{j=1}^{m} a_{i}^{j,k}x_{j}^{k}-x \right\|^2 \\
& =\sum_{j=1}^m a_{i}^{j,k} \|x_{j}^{k}-x\|^2  
- \frac{1}{2}\sum_{j=1}^m\sum_{\ell=1}^m a_{i}^{j,k}a_{i}^{\ell,k}\|x_j^k-x_\ell^k\|^2.\end{align*}
By summing these relations over $i$,  exchanging the order of summations, and using 
$\sum_{i=1}^m a_i^{j,k}=1$ for all $j$ and $k$ (cf.~Assumption~\ref{assump:doubstoch}(a)),  we obtain
for all $x\in X$ and all $k\ge0$,
\begin{align*}& \sum_{i=1}^m\|v_{i}^k -x\|^2 \\  
&=\sum_{j=1}^{m} \|x_{j}^{k}-x \|^2 
- \frac{1}{2}\sum_{j=1}^m\sum_{\ell=1}^m\sum_{i=1}^m a_{i}^{j,k}a_{i}^{\ell,k}\|x_j^k-x_\ell^k\|^2.
\end{align*}
By using the connectivity assumption on the graph $\mathcal{G}^k$ and the assumption on the entries in 
the matrix $A(k)$ (cf.~Assumptions~\ref{assump:doubstoch}(b) and $($c)), we can see that there exists a spanning tree
$\mathcal{T}^k\subseteq \mathcal{G}^k$ such that 
\[\frac{1}{2}\sum_{j=1}^m\sum_{\ell=1}^m\sum_{i=1}^m a_{i}^{j,k}a_{i}^{\ell,k}\|x_j^k-x_\ell^k\|^2
\ge \sum_{\{s,\ell\}\in \mathcal{T}^k} \eta^2 \|x_s^k-x_\ell^k\|^2.\]
Therefore, for all $x\in X$ and all $k\ge0$,
\[\sum_{i=1}^m\|v_{i}^k -x\|^2 
\le \sum_{j=1}^{m} \|x_{j}^{k}-x \|^2 - \eta^2 \sum_{\{s,\ell\}\in \mathcal{T}^k}  \|x_s^k-x_\ell^k\|^2,\]
and the stated relation follows by substituting the preceding relation in equation~\eqref{eq:mid0}.
\end{proof}

Our next lemma provides a relation for the iterates $\theta^k_i$ related to the learning scheme of the algorithm. 
\begin{lemma}\label{lemm:thetadesc}
Let Assumptions~\ref{assump:errormoment} and~\ref{assump:functionsh} hold,
and let the iterates $\theta^k_i$ be generated
by the algorithm~\eqref{aver}--\eqref{eq:projec}.  Then, almost surely,  
we have for all $k\ge 0$,
\begin{align*}&
\sum_{i=1}^{m}\bkE \left[\| \theta_i^{k+1} -\theta^* \|^2 | \mathcal{F}_k \right]  \cr&
\leq
\left(1  - 2\gamma_{k}\kappa + \gamma_{k}^2 R_{\theta}^2\right)
\sum_{i=1}^{m}\|\theta_i^{k} -\theta^* \|^2 + m\gamma_{k}^2\nu_{\theta}^2,
\end{align*}
with $\theta^*=\argmin_{\theta\in\Theta} h(\theta)$.
\end{lemma}

\begin{proof}
By using the nonexpansivity of the projection operator, the
strong monotonicity and Lipschitz continuity of $\nabla_{\theta}
h(\theta)$, and by recalling the
relation $\theta^*  =  \Pi_{\theta} \left[\theta^* - \gamma_k \nabla_{\theta}
h(\theta^*)\right], $
we obtain the following relation
\begin{align*}
&\|  \theta_i^{k+1}   -\theta^{*} \|^2\\
& \leq
\|\theta_i^{k}-\gamma_{k}(\nabla h(\theta_i^{k})+ \beta_i^k)
	-\theta^* + \gamma_k
	\nabla h(\theta^*)\|^2 \\
	& = \|\theta_i^{k} -\theta^* \|^2
+ \gamma_{k}^2 \|\nabla h(\theta_i^k) -\nabla h(\theta^*)\|^2
+ \gamma_{k}^2 \|\beta_i^k \|^2 \\
&- 2\gamma_{k}(\nabla h(\theta_i^{k})-\nabla h(\theta^*))^{T}(\theta_i^k-\theta^*) \\
& - 2\gamma_{k}(\beta_i^k)^T(\theta_i^k-\theta^* -\gamma_{k}(\nabla
			h(\theta_i^{k})-\nabla h(\theta^*))) \\
	& \leq (1-2\gamma_k \kappa + \gamma_k^2 R_{\theta}^2) \|\theta_i^k
	-\theta^*\|^2 + \gamma_{k}^2 \|\beta_i^k \|^2 \\
&  - 2\gamma_{k}(\beta_i^k)^T(\theta_i^k-\theta^* -\gamma_{k}(\nabla
			h(\theta_i^{k})-\nabla h(\theta^*))).
\end{align*}
Taking conditional expectation with respect to the past $\mathcal{F}_k$, we see that almost surely for all $i$ and $k$,
$$ \bkE \left[ \|\theta_i^{k+1}-\theta^*\|^2\mid {\cal F}_k\right]
	\leq (1-2\gamma_k \kappa + \gamma_k^2 R_{\theta}^2) 
	\|\theta_i^k -\theta^*\|^2  + \gamma_{k}^2 \nu_{\theta}^2,$$ 
	since $\mathbb{E}[(\beta_i^k)^T(\theta_i^k-\theta^* -\gamma_{k}(\nabla
			h(\theta_i^{k})-\nabla h(\theta^*))) \mid {\cal F}_k] = 0$ 
			and  $\mathbb{E}[\|\beta_i^k \|^2\mid {\cal F}_k] \le \nu_{\theta}^2$
			(by Assumption~\ref{assump:errormoment}). 
By summing the preceding relations over $i$, we obtain the stated result.
\end{proof}


The following lemma gives a key result 
that combines the decrease properties for $x$- and $\theta$-iterates 
established in Lemmas~\ref{lemm:xdesc} and~\ref{lemm:thetadesc}.
\begin{lemma}\label{lemma:xandtheta}
Let Assumptions~\ref{assump:doubstoch}--\ref{assump:functionsh} hold,
and let $X=\cap_{i=1}^m X_i$. 
Let the sequences $\{x^k_i\},\{\theta_i^k\}$, $i\in{\cal N}$, be generated
according to~\eqref{aver}--\eqref{eq:projec},  and define 
\[V(x^k,\theta^k;x):=\sum_{i= 1}^{m} \left( \|x_{i}^k
		-x\|^2  + \|\theta_i^{k} -\theta^* \|^2 \right) 
		\hbox{for all} x\in X.\]
		 Then, for all $x\in X$, all $k\ge 0$, and all $\ell\in{\cal N},$
the following relation holds almost surely
\begin{align*}
&\bkE [V(x^{k+1},\theta^{k+1};x) \mid \mathcal{F}_k]  
 \leq V(x^k,\theta^k;x)  + m\alpha_k^{2-\tau}L_{\theta}^2 D^2\cr
& -\eta^2\sum_{\{j,s\}\in{\cal T}^k} \|x_j^k-x_s^k\|^2 
 +m\alpha_k^2 (2S^2 + \nu^2) \cr
& + 2\alpha_k G \sum_{j=1}^m\|x_j^k - z^k\| 
 -2\alpha_k\left( f(z^k,\theta^*) -f(x,\theta^*) \right)\cr
& -\gamma_k\left(2\kappa - \gamma_{k} R_{\theta}^2 -\frac{\alpha_k^{\tau} + 2\alpha^2_kL_{\theta}^2}{\g_k}\right)   
\sum_{i = 1}^{m} \|\theta_i^{k} -\theta^* \|^2 \\
& + m\gamma_{k}^2\nu_{\theta}^2,
\end{align*}
where 
\[z^k=\Pi_X[y^k] \quad\hbox{with } y^k=\frac{1}{m}\sum_{j=1}^m x_j^k\quad \qquad\hbox{for all }k\ge0,\]
	$\mathcal{T}^k$ denotes a spanning tree in $\mathcal{G}^k,$
while $S, \bar X$
and $G$ are defined as 
$S\triangleq\max_i\max_{x\in \bar X}\| \nabla_x f_i(x,\theta^*) \|$, 
$\bar X \triangleq \mbox{conv}(\cup_{i=1}^m X_i)$, and
$G\triangleq\max_{i\in{\cal N}}\max_{z_\in X}\|\nabla_x
	f_i(z,\theta^*)\|$.

\end{lemma}
\begin{proof} 
By Lemma~\ref{lemm:xdesc} we have almost surely for some $\tau>0$ and 
for all $x\in X$ and all $k\ge 0$,
\begin{align*}
& \sum_{i = 1}^{m}  \bkE  \left[ \|x_{i}^{k+1} -x \|^2 \mid \mathcal{F}_k\right]  \cr 
& \leq \sum_{j = 1}^{m} \|x_{j}^k -x\|^2  
- \eta^2 \sum_{\{s,\ell\}\in \mathcal{T}^k}  \|x_s^k-x_\ell^k\|^2\cr
& +m\alpha_k^2 (2S^2 + \nu^2)
 + m\alpha_k^{2-\tau}L_{\theta}^2 D^2\cr
& + \left(\alpha_k^{\tau} + 2\alpha^2_kL_{\theta}^2\right) 
\sum_{i =1}^{m}\|\theta_i^k-\theta^*\|^2 \\ 
& -2\alpha_k \sum_{i = 1}^{m}\left(f_i(v_i^k,\theta^*)-f_i(x,\theta^*)\right),
\end{align*}
where $\mathcal{T}^k$ is a spanning tree in the graph $\mathcal{G}^k$ and $\theta^*=\argmin_{\theta\in\Theta} h(\theta)$,
which exists  and it is unique in view of the strong convexity of $h$.
By Lemma~\ref{lemm:thetadesc} almost surely we have for all $k\ge 0$,
\begin{align*}
& \quad \sum_{i=1}^{m}\bkE \left[\| \theta_i^{k+1} -\theta^* \|^2 | \mathcal{F}_k \right] \\
& \leq
\left(1  - 2\gamma_{k}\kappa + \gamma_{k}^2 R_{\theta}^2\right)
\sum_{i=1}^{m}\|\theta_i^{k} -\theta^* \|^2 + m\gamma_{k}^2\nu_{\theta}^2.
\end{align*}
By combining Lemmas~\ref{lemm:xdesc} and~\ref{lemm:thetadesc} and using
the notation $V$,
after regrouping some terms, we see that for all $x\in X$ and
all $k\ge 0$, {the following holds almost surely:} 
\begin{align*}
& \bkE  \left[ V(x^{k+1},\theta^{k+1};x) \mid \mathcal{F}_k\right] 
\leq  V(x^k,\theta^k;x)+ m\alpha_k^{2-\tau}L_{\theta}^2 D^2 
\cr
& - \eta^2 \sum_{\{s,\ell\}\in \mathcal{T}^k}  \|x_s^k-x_\ell^k\|^2+m\alpha_k^2 (2S^2 + \nu^2)
 \cr
 & + \left(\alpha_k^{\tau} + 2\alpha^2_kL_{\theta}^2\right) 
\sum_{i =1}^{m}\|\theta_i^k-\theta^*\|^2
\end{align*}
\begin{align} \label{eq:le4r1}
& -2\alpha_k \sum_{i =
	1}^{m}\left(f_i(v_i^k,\theta^*)-f_i(x,\theta^*)\right)\notag \\ 
& -\gamma_k\left(2\kappa - \gamma_{k} R_{\theta}^2\right)  
\sum_{i = 1}^{m} \|\theta_i^{k} -\theta^* \|^2
+ m\gamma_{k}^2\nu_{\theta}^2.
\end{align}
Next, we work with the term involving the function values. We consider the summand  
$f_i(v_i^k,\theta^*)-f_i(x,\theta^*)$. Define 
\[y^k=\frac{1}{m}\sum_{j=1}^m x_j^k,\quad z^k=\Pi_X[y^k]\qquad\hbox{for all }k\ge0.\]
By adding and subtracting $f_i(z^k)$,
and by using the convexity of $f_i(\cdot,\theta^*)$ (see Assumption~\ref{assump:functionsf}), 
we can show that 
\begin{align*}
& \quad f_i(v_i^k,\theta^*)-f_i(x,\theta^*) \\
& \ge    f_i(z^k,\theta^*)^T (v_i^k - z^k) + f_i(z^k,\theta^*)-f_i(x,\theta^*)\cr
&  \ge    - \|\nabla_x f_i(z^k,\theta^*)\| \|v_i^k - z^k\| + f_i(z^k,\theta^*)-f_i(x,\theta^*).
\end{align*}
Since $X_\ell$ is bounded for every $\ell$ and $z^k\in X=\cap_{i=1}^m X_i$, it follows that  
\[\max_{k\ge0}
\|\nabla_x f_i(z^k,\theta^*)\|\le 
G\triangleq\max_{i\in{\cal N}}\left(\max_{y\in X}\|\nabla_x f_i(y,\theta^*)\|\right).\]
Thus, $\|\nabla_x f_i(z^k,\theta^*)\| \|v_i^k - z^k\|
\le G \|v_i^k - z^k\|$, implying 
\begin{align*}
& \quad \sum_{i=1}^m \left(f_i(v_i^k,\theta^*)-f_i(x,\theta^*)\right) \\
& \ge  - G  \sum_{i=1}^m\|v_i^k - z^k\| + f(z^k,\theta^*) - f(x,\theta^*),
\end{align*}
where we also use notation $f(\cdot,\theta)=\sum_{i=1}^m f_i(\cdot,\theta)$.
Recalling the definition of $v_i^k$, and by using the doubly-stochastic property of the weights and
 the convexity of the Euclidean norm, we can see that
$\sum_{i=1}^m\|v_i^k - z^k\|\le \sum_{i=1}^m\sum_{j=1}^m a_i^{j,k}\|x_j^k - z^k\|=\sum_{j=1}^m \|x_j^k - z^k\|.$
Hence,
\begin{align}\label{eq:le4r2}
& \quad \sum_{i=1}^m
\left(f_i(v_i^k,\theta^*)-f_i(x,\theta^*)\right)\notag \\
& \ge  - G  \sum_{j=1}^m\|x_j^k - z^k\| 
+ f(z^k,\theta^*) - f(x,\theta^*).
\end{align}
Using~\eqref{eq:le4r2} in inequality~\eqref{eq:le4r1} yields the stated result.
\end{proof}

\subsection{Averages and Constraint Sets Intersection}\label{sec:aver}
Now, we focus on developing a relation that will be useful for providing a bound 
on the distance of the iterate averages $y^k=\frac{1}{m}\sum_{j=1}^m x_j^k$ and the intersection set $X=\cap_{i=1}^m X_i$. Specifically, 
the goal is to have a bound
for  $ \sum_{j=1}^m\|x_j^k - z^k\|$, which will allow us to leverage on Lemma~\ref{lemma:xandtheta} and prove 
the almost sure convergence of the method. 
We provide such a bound for generic points $x_1,\ldots,x_m$ taken from sets $X_1,\ldots, X_m$, respectively.
For this, we strengthen Assumption~\ref{assump:bnd}(b) on the sets $X_i$
by requiring  that the interior of $X$ is nonempty. This assumption has also been used in~\cite{nedich10consensus}
to ensure that the iterates $x_i^k\in X_i$ have accumulation points in $X$. 
This assumption and its role
in such set dependent iterates has been originally illuminated in~\cite{GPR67}.

\begin{assumption}
\label{assump:slater}
There exists a vector $\bar{x} \in {\rm int}(X)$, i.e., there exists a scalar $\delta > 0$ such that
$\left\{z \hspace{1mm}|\hspace{1mm} \|z - \bar x\|  \leq \delta \right\} \subset X$.
\end{assumption}
By using Lemma 2(b) from~\cite{nedich10consensus}
and boundedness of the sets $X_i$, we establish  an upper bound for
$\sum_{j=1}^m \|x_j - \Pi_X[\frac{1}{m}\sum_{\ell=1}^m x_\ell]\|$ for arbitrary points $x_i\in X_i$, 
as given in the following lemma.
\begin{lemma}\label{lemma:xaverbound}
Let Assumptions~\ref{assump:bnd} and~\ref{assump:slater} hold.
Then, for the vector $\hat x = \frac{1}{m}\sum_{\ell=1}^m x_\ell$, with $x_\ell\in X_\ell$ for all $\ell$,
we have
\[\sum_{j=1}^m \|x_j - \Pi_X[\hat x]\|
\le m \left(1+\frac{mD}{\delta}\right) \max_{j,\ell\in{\cal N}}\|x_j - x_\ell\|.\]
\end{lemma}

Under the interior-point assumption, we provide a refinement of
Lemma~\ref{lemma:xandtheta}, which will be the key relation for
establishing the convergence.
\begin{lemma}\label{lemma:key}
Let Assumptions~\ref{assump:doubstoch}--\ref{assump:slater} hold, and
let $X=\cap_{i=1}^m X_i$. 
Let the sequences $\{x^k_i\},\{\theta_i^k\}$
be generated according to~\eqref{aver}--\eqref{eq:projec}.  
Then, 
almost surely, we have for all $x\in X$, all $k\ge 0$, and all $\ell\in{\cal N},$
\begin{align*}
&\quad \bkE [V(x^{k+1},\theta^{k+1};x) \mid \mathcal{F}_k]  
 \leq V(x^k,\theta^k;x) \cr
&- \left(\frac{\eta^2}{m-1} - \a_k^\sigma\right)\max_{j,s\in \mathcal{N} }  \|x_j^k-x_s^k\|^2 \cr
& +m\alpha_k^2 (2S^2 + \nu^2) + m\alpha_k^{2-\tau}L_{\theta}^2 D^2\cr
& + \alpha_k^{2-\sigma} G^2 m^2 \left(1+\frac{mD}{\delta}\right)^2
 -2\alpha_k\left( f(z^k,\theta^*) -f(x,\theta^*) \right)\cr
& -\gamma_k\left(2\kappa - \gamma_{k} R_{\theta}^2 -\frac{\alpha_k^{\tau} + 2\alpha^2_kL_{\theta}^2}{\g_k}\right)   
\sum_{i = 1}^{m} \|\theta_i^{k} -\theta^* \|^2\cr
&+ m\gamma_{k}^2\nu_{\theta}^2,
\end{align*}
where $\sigma>0$, while $z^k$, $y^k$ and other variables and constants are the same as in Lemma~\ref{lemma:xandtheta}.
\end{lemma}

\section{Almost sure convergence}\label{sec:global}
We now prove the almost sure convergence of the sequences
	produced by the algorithm for suitably selected stepsizes $\a_k$ and $\g_k$.
In particular, we impose the following requirements on the stepsizes.
\begin{assumption}[Stepsize sequences]
\label{assump:stepseq}
The steplength sequences $\{\alpha_k\}$ and $\{\gamma_k\}$
satisfy the following conditions :
$$ \sum_{k=0}^{\infty} \gamma_k = \infty,\qquad
		\sum_{k=0}^{\infty} \gamma_k^2 < \infty, \qquad \sum_{k=0}^{\infty} \alpha_k = \infty,$$
and for some $\tau \in(0,2),$
$$\sum_{k=0}^{\infty} \alpha_k^{2-\tau} < \infty,\qquad
\lim_{k \rightarrow \infty}\frac{\alpha_k^{\tau}}{\gamma_k} = 0.$$
\end{assumption}

\textbf{Example for the stepsizes:} 
A set of choices satisfying the above assumptions  are 
$\gamma_k = k^{-a_1}$ and $\alpha_k = k^{-a_2}$, where 
\begin{itemize}
\item $1 > a_2 > a_1 > \frac{1}{2}$; 
	\item $a_2 (2-\tau) > 1 \implies \tau < 2 - 1/a_2$; 
	\item $a_1 < \tau a_2 \implies \tau > a_2 /a_1.$
\end{itemize}
There is an
	infinite set of choices for $(a_1,a_2,\tau)$ that satisfy these conditions; a concrete example is
$(a_1,a_2,\tau) =  (0.51,0.9,0.75)$. Note that $a_2 > a_1$ implies that
	the steplength sequence employed in computation decays faster than
	the corresponding sequence of the learning updates.  

To analyze the behavior of the  sequences $\{\theta_i^k\}$, $i\in{\cal N}$,
we leverage the following super-martingale convergence result  
from~\cite[Lemma~10, page~49]{polyak87introduction}.
\begin{lemma}\label{lemm:expseq1}
Let $\{v_k\}$ be a sequence of
nonnegative random variables adapted to $\sigma$-algebra $\tilde{\mathcal{F}}_k$ and
such that almost surely
$$\bkE[v_{k+1} \mid \tilde{\mathcal{F}_{k}}] \leq  (1-u_k) v_k + \psi_k\quad \hbox{for all }k \geq 0, $$
where $0 \leq u_k \leq 1, \psi_k \geq  0$, 
$\sum_{k=0}^{\infty} u_k = \infty,
\sum_{k=0}^{\infty} \psi_k < \infty,$ and $\lim_{k \to \infty} \frac{\psi_{k}}{u_{k}} = 0.$
Then, almost surely  $\lim_{k\to\infty} v_k =0$.
\end{lemma}

Next, we establish a convergence property for the $\theta$-iterates of the algorithm. 
\begin{proposition}[{Almost sure convergence of $\{\theta_i^k\}$}]\label{prop:thetaconv}
Let Assumptions~\ref{assump:errormoment} and~\ref{assump:functionsh} hold. Also, let $\g_k$ satisfy the conditions of Assumption~\ref{assump:stepseq}.
Let the iterates $\theta^k_i$ be generated
according to~\eqref{aver}--\eqref{eq:projec}. If  
$\theta^*=\argmin_{\theta\in\Theta} h(\theta)$, then $\theta_i^k \to
\theta^*$ as $k \to \infty$ in an almost sure sense for $i = 1, \ldots,
	N.$\end{proposition}
\begin{proof} We provide a brief proof.
By Lemma~\ref{lemm:thetadesc} almost surely for all $k\ge0$, 
\begin{align*}
& \quad \sum_{i=1}^{m}\bkE \left[\| \theta_i^{k+1} -\theta^* \|^2 \mid
\mathcal{F}_k \right]  \cr & \leq
\left(1  - 2\gamma_{k}\kappa + \gamma_{k}^2 R_{\theta}^2\right)
\sum_{i=1}^{m}\|\theta_i^{k} -\theta^* \|^2 + m\gamma_{k}^2\nu_{\theta}^2.
\end{align*}
Using Assumption~\ref{assump:stepseq}, we can show that
for all $k\ge \hat k$, $\gamma_k \leq \frac{\kappa}{R_{\theta}^2}$. Then, 
we have almost surely
\begin{align*}
\sum_{i=1}^{m}\bkE \left[\| \theta_i^{k+1} -\theta^* \|^2 \mid
\mathcal{F}_k \right] & \leq
\left(1  - \gamma_{k}\kappa\right)
\sum_{i=1}^{m}\|\theta_i^{k} -\theta^* \|^2\\
		&+ m\gamma_{k}^2\nu_{\theta}^2.
\end{align*}
To invoke Lemma~\ref{lemm:expseq1}, we define 
$v_k = \sum_{i=1}^{m}\| \theta_i^{k} -\theta^* \|^2$.
Furthermore, $u_k= \gamma_{k}\kappa$ and
$\psi_k=m\gamma_{k}^2\nu_{\theta}^2$ for all $k\ge0$. 
We note that $\sum_{k\ge0} u_k=\infty$,  
$\sum_{k=0}^{\infty} \psi_k < \infty,$ and $\lim_{k \to \infty} \frac{\psi_{k}}{u_{k}} = 0$ by 
Assumption~\ref{assump:stepseq}.
Thus, Lemma~\ref{lemm:expseq1} applies to a shifted sequence
$\{v_k\}_{k\ge \hat k}$
and we conclude that $v_k\to 0$ almost surely.
\end{proof}

Now, we analyze the behavior of $x$-sequences, where we leverage the following 
super-martingale convergence theorem  from~\cite[Lemma~11, page~50]{polyak87introduction}.
\begin{lemma}\label{lemm:expseq2}
Let $v_k, u_k, \psi_k$ and $\delta_k$ be nonnegative random variables adapted to
a $\sigma$-algebra $\tilde {\mathcal{F}}_k$. If almost  surely
$\sum_{k=0}^{\infty} u_k < \infty, \sum_{k=0}^{\infty}\psi_k < \infty$, and
$$\bkE[v_{k+1} \mid \tilde{\mathcal{F}}_k] \leq (1+u_k)v_k - \delta_k + \psi_k\hspace{2mm} \hbox{for all }k \geq 0,$$
then almost surely\ $V_k$ is convergent and
$\sum_{k=0}^{\infty} \delta_k < \infty$.
\end{lemma}

As observed in Section~\ref{sec:assump}, under continuity of the functions $f_i(\cdot,\theta)$ 
(in view of Assumption~\ref{assump:functionsf}) 
and the compactness
of each $X_i$, the problem $\min_{x\in\cap_{i=1}^m X_i} \sum_{i=1}^m f_i(x,\theta^*)$ has a solution. 
We denote the set of solutions by $X^*$.
Therefore, under Assumptions~\ref{assump:bnd}, \ref{assump:functionsf}, and~\ref{assump:functionsh}, 
the problem~\eqref{eq:prob}
 has a nonempty solution set, given by $X^*\times\{\theta^*\}$.

We have the following convergence result.

\begin{proposition}[Almost sure convergence of $\{x_i^k\}$]\label{prop:xconv}
Let Assumptions~\ref{assump:doubstoch}--\ref{assump:stepseq} hold, and let $X=\cap_{i=1}^m X_i$.
Let the sequences $\{x^k_i\},\{\theta_i^k\}$ be generated
according to~\eqref{aver}--\eqref{eq:projec}.  
Then, the sequences $\{x_j^k\}$ converge almost surely to the same solution point, i.e.,
there exists a random vector $z^*\in X^*$ such that almost surely
\begin{align*}
\lim_{k\to\infty} x_j^k=z^*\qquad\hbox{for all }j\in {\cal N}. 
\end{align*}
\end{proposition}

\begin{proof}
In Lemma~\ref{lemma:xandtheta}, we let $x$
be an optimal solution for the problem $\min_{x\in\cap_{i=1}^m X_i} \sum_{i=1}^m f_i(x,\theta^*)$, i.e.,
$x=x^*$ with $x^*\in X^*$. Thus, by Lemma~\ref{lemma:xandtheta} we obtain almost surely
for any $x^*\in X^*$, all $k\ge 0$, and all $\ell\in{\cal N},$
\begin{align*}
& \quad \bkE [{V(x^{k+1},\theta^{k+1};x^*)} \mid \mathcal{F}_k] \\ 
& \leq V(x^k,\theta^k;x^*) - \left(\frac{\eta^2}{m-1} - \a_k^\sigma\right)\max_{j,s\in \mathcal{N} }  \|x_j^k-x_s^k\|^2 \cr
& +m\alpha_k^2 (2S^2 + \nu^2) + m\alpha_k^{2-\tau}L_{\theta}^2 D^2 \cr
&+ \alpha_k^{2-\sigma} G^2 m^2 \left(1+\frac{mD}{\delta}\right)^2
 -2\alpha_k\left( f(z^k,\theta^*) - f(x^*,\theta^*) \right)\cr
& -\gamma_k\left(2\kappa - \gamma_{k} R_{\theta}^2 -\frac{\alpha_k^{\tau} + 2\alpha^2_kL_{\theta}^2}{\g_k}\right)   
\sum_{i = 1}^{m} \|\theta_i^{k} -\theta^* \|^2 \\
& + m\gamma_{k}^2\nu_{\theta}^2.
\end{align*}
Next, since $\sigma$ is an arbitrary positive scalar, we let
$\sigma=\tau$ where $\tau \in (0,2)$ is obtained from
Assumption~\ref{assump:stepseq}. Furthermore,  let $\psi_k$ be
	defined as follows:
\begin{align*}
	\psi_k & \ {\triangleq } \ m\alpha_k^2 (2S^2 + \nu^2)+m\gamma_{k}^2\nu_{\theta}^2 \\
 & + \alpha_k^{2-\tau}\left( mL_{\theta}^2 D^2 +  G^2 m^2
		 \left(1+\frac{mD}{\delta}\right)^2\right).\end{align*}

		 		 
Using the assumptions on the stepsizes we can show that 
for all $k\ge k_0$, we have
$\frac{\eta^2}{m-1} - \a_k^{{\tau}} \ge\epsilon$ and  
\[2\kappa - \gamma_{k} R_{\theta}^2 -\frac{\alpha_k^{\tau} + 2\alpha^2_kL_{\theta}^2}{\g_k}\ge 0.\]
Therefore, almost surely for all $x^*\in X^*$, all $k\ge k_0$, and all $\ell\in{\cal N},$
\begin{align*}
&\bkE [V(x^{k+1},\theta^{k+1};x^*) \mid \mathcal{F}_k]  
\leq V(x^k,\theta^k;x^*) \cr
& - \epsilon \max_{j,s\in \mathcal{N} }  \|x_j^k-x_s^k\|^2 
 -2\alpha_k\left( f(z^k,\theta^*) -f(x^*,\theta^*) \right) 
+\psi_k.
\end{align*}
Recall that $z^k=\Pi_X[y^k]$ with $y^k=\frac{1}{m}\sum_{\ell=1}^m x_\ell^k$. In view of optimality of $x^*$, 
we have $f(z^k,\theta^*) -f(x^*,\theta^*)\ge0$ for all $k$ and $x^*\in X^*$. Furthermore, the conditions on the stepsizes in Assumption~\ref{assump:stepseq} 
Then, we verify that the conditions of Lemma~\ref{lemm:expseq2} are satisfied for
the sequence {$\{V(x^k,\theta^k;x^*)\}_{k \geq k_0}$}
for an arbitrary $x^*\in X^*$. By Lemma~\ref{lemm:expseq2} it follows that
{$V(x^k,\theta^k;x^*)$} is convergent almost surely for every $x^*\in
X^*$, and the following hold almost surely:
\begin{align}
& \sum_{k=0}^\infty  \max_{j,s\in \mathcal{N} }  \|x_j^k-x_s^k\|^2 <\infty,\label{eq:sumx} \\
& \sum_{k=0}^\infty \alpha_k\left( f(z^k,\theta^*)
		-f(x^*,\theta^*) \right) <\infty.\label{eq:sumf}
		\end{align}

By Proposition~\ref{prop:thetaconv}, we have that $\theta_i^k\to\theta^*$ almost surely for all $i\in{\cal N}$. 
Since {$V(x^k,\theta^k;x^*)=\sum_{i= 1}^{m} \left( \|x_{i}^k
		-x^*\|^2  + \|\theta_i^{k} -\theta^* \|^2 \right)$ and}
the assertion that 
$\{V(x^k,\theta^k;x^*)\}$ is convergent almost surely  for every $x^*\in X^*$, we can conclude that 
\begin{align}\label{eq:xseq}
\hbox{$\left\{\sum_{i= 1}^{m} \|x_{i}^k -x^*\|^2\right\}$ is convergent a.s.
	 $\ \forall \ x^*\in X^*$}.
\end{align}
Since $\sum_{k=0}^\infty \a_k=\infty$, the relation~\eqref{eq:sumf}
implies that 
\begin{align}\liminf_{k\to\infty} f(z^k,\theta^*) =f^*,\end{align}
where $f^*$ is the optimal value of the problem, i.e., $f^*=f(x^*,\theta^*)$ for any $x^*\in X^*$.
The set $X$ is bounded (since each $X_j$ is bounded by assumption),  so the sequence $\{z^k\}\subset X$
is also bounded. Let ${\cal K}$ denote the index set of a
subsequence along which the following holds almost surely:
\[\lim_{k\to\infty,k\in{\cal K}} f(z^k,\theta^*) = \liminf_{k\to\infty} f(z^k,\theta^*),\]
\begin{align}\label{eq:zk}
\lim_{k\to\infty,k\in{\cal K}} z^k=z^*\qquad\hbox{with }z^*\in X^*.\end{align}
We note that ${\cal K}$ is a random sequence and $z^*$ is a {randomly
	specified vector from $X^*$}.  Further, relation~\eqref{eq:sumx}
	 implies that 
all the sequences $\{x_j^k\}$, $j=1,\ldots,m$, have the same accumulation points (which exist since the sets $X_j$ are bounded). 
Moreover, since $\{x_j^k\}\subset X_j$ for each $j\in{\cal N}$, it follows that 
the accumulation points of the sequences $\{x_j^k\}$, $j=1,\ldots,m$, must lie in the set $X=\cap_{j=1}^m X_j$.
Without loss of generality we may assume that the limit $\lim_{k\to\infty,k\in{\cal K}} x_j^k$ exists a.s.\ for each $j$,
so that in view of the preceding discussion we have almost surely
\[\lim_{k\to\infty,k\in{\cal K}} x_j^k = \tilde x,\qquad\hbox{with }\tilde x\in X,\]
\[\lim_{k\to\infty,k\in{\cal K}} y^k =\lim_{k\to\infty,k\in{\cal K}} \frac{1}{m}\sum_{\ell=1}^m x_\ell^k = \tilde x.\]
Then, by the continuity of the projection operator $v\mapsto \Pi_X[v]$ and the fact $z_k=\Pi_X[y^k]$, we have almost surely
\[\lim_{k\to\infty,k\in{\cal K}} z^k=\lim_{k\to\infty,k\in{\cal K}} \Pi_X[y^k]=\tilde x.\]
The preceding relation and~\eqref{eq:zk} yield $\tilde x=z^*$, implying that for all $j$ almost surely
\begin{align}\label{eq:partx}
\lim_{k\to\infty,k\in{\cal K}} x_j^k = z^*,\qquad\hbox{with }z^*\in X^*.\end{align}
Then, we can use $x^*=z^*$ in relation~\eqref{eq:xseq} to conclude that 
$\sum_{i= 1}^{m} \|x_{i}^k -z^*\|^2$ is convergent almost surely. This and the subsequential convergence in~\eqref{eq:partx}
imply that $\sum_{i= 1}^{m} \|x_{i}^k -z^*\|^2\to 0$ almost surely. 
\end{proof}
\noindent {\bf Special cases:} 
We note two  special cases of relevance which arise as a consequence of Propositions~\ref{prop:thetaconv} and~\ref{prop:xconv}.

\noindent { (i) \bf Deterministic optimization and learning:} First, note that if the functions $f_i(x,\theta)$ and $h(\theta)$ are deterministic in that 
the gradients $\nabla_x f(x,\theta)$ and $\nabla h(\theta)$ may be evaluated  at arbitrary
points $x$ and $\theta$, then the results of Propositions~\ref{prop:thetaconv} and~\ref{prop:xconv}
show that $\lim_{k\to\infty} \theta_i^k= \theta^*$ and $\lim_{k\to\infty} x_i^k= x^*$ for some $x^*\in X^*$
and for all $i=1,\ldots,m$.

\noindent { (ii) \bf Correctly specified problems:} Second, now suppose that the parameter $\theta^*$ is known to every agent, so there is no misspecification. 
This case can be treated under algorithm~\eqref{aver}--\eqref{eq:projec} where the iterates $\theta_i^k$ are all fixed at 
$\theta^*$. Formally this can be done by setting the initial parameters to the correct value,
i.e., $\theta_i^0=\theta^*$ for all $i$, and by using the fact that the function $h(\theta^*)$ is known, in which case 
the algorithm reduces to: for all $i = 1, \hdots, m$ and $k\ge0$,
\begin{align}\label{aver2}
& v_{i}^{k}  := \sum_{j=1}^{m}a_{i}^{j,k}x_{j}^{k}\\ 
& x_{i}^{k+1}  := \Pi_{{X_i}} \left( v_{i}^{k} - {\alpha_k} \left( \nabla_x
			f_{i} (v_{i}^{k},\theta^*) + w_i^{k} \right)\right).
	\label{proj-x2}  
\end{align}
By letting $F_i(x)=f_i(x,\theta^*)$ we see that, by Proposition~\ref{prop:xconv}, the iterates of 
the algorithm~\eqref{aver2}--\eqref{proj-x2} converge almost surely to a solution of problem 
$\min_{x\in\cap_{i=1}^m X_i}\sum_{i=1}^m F_i(x)$. Thus, the algorithm solves this problem in a distributed fashion,
where both functions and the sets are distributed among the agents. In particular, this result when reduced to a deterministic case
(i.e., noiseless gradient evaluations) extends the convergence results established in~\cite{nedich10consensus} where
two special cases have been studied; namely, the case when $X_i=X$ for all $i$, and the case when the underlying graph is a complete graph and all weights are equal (i.e., $a_i^{j,k}=\frac{1}{m}$ for all $i,j$ and $k\ge0$).

\noindent {\bf Rate of convergence:} While standard stochastic
gradient methods achieve the optimal rate of convergence in that
$\mathbb{E}[f(x_k,\theta^*) ] - \mathbb{E}[f(x^*,\theta^*)] \leq {\cal
	O}(1/k)$ in the correctly specified regime, it remains to establish
	similar rates in this instance particularly in the context of
	time-varying connectivity graphs. Such rate bounds will aid in
	developing practical implementations.
\section{Concluding Remarks}\label{sec:concl}
Traditionally, optimization algorithms have been developed under the premise of
exact information regarding functions and constraints. As systems grow in
complexity, an a priori knowledge of cost functions and efficiencies is
difficult to guarantee. One avenue lies in using observational information to
learn these functions while optimizing the overall system. We consider
precisely such a question in a networked multi-agent regime where an agent does
not have access to the decisions of the entire collective, and are furthermore locally constrained by their own feasibility sets. Generally, in such
regimes, distributed optimization can be carried out by combining a local
averaging step with a projected gradient step. We overlay a learning step where
agents update their belief regarding the misspecified parameter and examine the
associated schemes in this regime. It is shown that when agents are
characterized by merely convex, albeit misspecified, problems under general
time-varying graphs, the resulting schemes produce sequences that converge
almost surely to the set of optimal solutions and the true parameter,
respectively.  

\bibliographystyle{ieeetr}
\bibliography{bib_egp-1}

\end{document}